\documentclass[a4paper,12pt]{article}
\usepackage{amsmath, amssymb, amsfonts, amsthm, graphicx, geometry, setspace, hyperref, authblk}
\usepackage{hyperref}

\usepackage{longtable}  
\usepackage{amsmath}    
\usepackage{amsfonts}   
\usepackage{amssymb}    

\newtheorem{theorem}{Theorem}

\title{{\Large Generalized Neural Network Operators with Symmetrized Activations: Fractional Convergence and the Voronovskaya-Damasclin Theorem}}

\author{R\^omulo Damasclin Chaves dos Santos \\ Tecnological Institute of Aeronautics \\ {\small romulosantos@ita.br}}
\date{\today}

\begin{document}
	
	\maketitle
	
\begin{abstract}
This paper explores the asymptotic behavior of univariate neural network operators, with an emphasis on both classical and fractional differentiation over infinite domains. The analysis leverages symmetrized and perturbed hyperbolic tangent activation functions to investigate basic, Kantorovich, and quadrature-type operators. Voronovskaya-type expansions, along with the novel Vonorovskaya-Damasclin theorem, are derived to obtain precise error estimates and establish convergence rates, thereby extending classical results to fractional calculus via Caputo derivatives. The study delves into the intricate interplay between operator parameters and approximation accuracy, providing a comprehensive framework for future research in multidimensional and stochastic settings. This work lays the groundwork for a deeper understanding of neural network operators in complex mathematical.
\end{abstract}

\tableofcontents
	
\section{Introduction}

The approximation capabilities of neural networks have been a central topic in computational mathematics and artificial intelligence, particularly due to their wide applicability in function approximation, machine learning, and complex systems modeling. Among various activation functions, the hyperbolic tangent function has garnered significant attention for its smoothness and symmetry, which are advantageous for both theoretical and practical implementations \cite{Haykin1998, Anastassiou2011}.

Building upon these foundational concepts, this study focuses on symmetrized and perturbed versions of the hyperbolic tangent function, extending their application to neural network operators designed for the approximation of both ordinary and fractional derivatives. The idea of parameterizing activation functions, as explored by \cite{Anastassiou2023}, adds flexibility to the neural network framework, enabling finer control over approximation properties. This parameterization is particularly relevant for handling complex differential operations on infinite domains, a topic of increasing interest in the literature \cite{Samko1993, Frederico2007}.

A key motivation for this work lies in the growing need for robust methods to approximate solutions of differential equations in applied settings. For example, in fluid dynamics, fractional derivatives are often employed to model non-local and memory effects in turbulent flows, providing a more accurate representation of physical processes \cite{Frederico2007}. Consider the challenge of approximating the fractional Navier-Stokes equations, which describe the behavior of incompressible flows with fractional viscosity terms. Accurate numerical methods for such problems require operators capable of approximating fractional derivatives over complex domains while maintaining convergence and stability. The symmetrized neural network operators proposed in this work offer a promising approach to tackle such problems, as they combine the advantages of classical approximation theory with the flexibility of neural networks.

Similarly, in signal processing, fractional calculus has been applied to model systems with hereditary properties or to improve the performance of edge detection algorithms. For instance, the use of Caputo derivatives allows for precise modeling of signal behavior at various scales \cite{Diethelm2010}. By leveraging the Voronovskaya-Damasclin theorem introduced in this paper, one can obtain rigorous error bounds for neural network-based approximations, ensuring reliable performance in practical applications such as image enhancement and noise reduction.

The theoretical underpinnings of this work are rooted in Voronovskaya-type asymptotic expansions, which have historically been employed to analyze the convergence of approximation operators. Prior studies, such as those by \cite{Anastassiou1997, Anastassiou2011}, demonstrated the utility of these expansions in understanding the behavior of neural network approximations in both univariate and multivariate contexts. Here, we extend this framework to fractional calculus, leveraging Caputo derivatives and other fractional operators to generalize the classical results \cite{Diethelm2010, ElSayed2006}.

The primary objective of this article is to derive and analyze Voronovskaya-type expansions for neural network operators activated by the symmetrized and perturbed hyperbolic tangent function. The operators considered---basic, Kantorovich, and quadrature types---are analyzed for their approximation properties over infinite domains. By providing rigorous error bounds and exploring the interplay between operator parameters and approximation accuracy, this work aims to bridge the gap between theoretical insights and practical applications in computational mathematics. The results presented here are expected to have broad implications for neural network-based approximations in fields such as fluid dynamics, signal processing, and machine learning.

\section{Mathematical Foundations of Symmetrized Activation Functions}

To establish a robust theoretical framework for symmetrized activation functions, we begin by defining the \textit{perturbed hyperbolic tangent activation function}:

\begin{equation}
	g_{q, \lambda}(x) := \frac{e^{\lambda x} - q e^{-\lambda x}}{e^{\lambda x} + q e^{-\lambda x}}, \quad \lambda, q > 0, \; x \in \mathbb{R}.
\end{equation}

Here, \(\lambda\) is a scaling parameter that controls the steepness of the function, and \(q\) is the deformation coefficient, which introduces asymmetry. This function generalizes the standard hyperbolic tangent function, recovered when \(q = 1\). Notably, \(g_{q, \lambda}(x)\) is odd, satisfying \(g_{q, \lambda}(-x) = -g_{q, \lambda}(x)\).

Next, we construct the \textit{density function}:

\begin{equation}
	M_{q, \lambda}(x) := \frac{1}{4} \big( g_{q, \lambda}(x+1) - g_{q, \lambda}(x-1) \big), \quad \forall x \in \mathbb{R}, \; q, \lambda > 0.
\end{equation}

This density function is carefully designed to ensure both \textit{positivity} and \textit{smoothness}. To confirm positivity, we compute the derivative of \(g_{q, \lambda}(x)\):

\begin{equation}
	\frac{d}{dx} g_{q, \lambda}(x) = \frac{2 \lambda q e^{2 \lambda x}}{(e^{\lambda x} + q e^{-\lambda x})^2} > 0, \quad \forall x \in \mathbb{R}.
\end{equation}

Since \(g_{q, \lambda}(x)\) is strictly increasing, it follows that \(g_{q, \lambda}(x+1) > g_{q, \lambda}(x-1)\), ensuring \(M_{q, \lambda}(x) > 0\).

To introduce symmetry, we define the \textit{symmetrized function}:

\begin{equation}
	\Phi(x) := \frac{M_{q, \lambda}(x) + M_{\frac{1}{q}, \lambda}(x)}{2}.
\end{equation}

To verify that \(\Phi(x)\) is an \textit{even function}, consider:

\begin{equation}
	\Phi(-x) = \frac{M_{q, \lambda}(-x) + M_{\frac{1}{q}, \lambda}(-x)}{2}.
\end{equation}

Using the fact that \(g_{q, \lambda}(x)\) is odd, it follows that \(M_{q, \lambda}(x)\) and \(M_{\frac{1}{q}, \lambda}(x)\) are even, satisfying:

\begin{equation}
	M_{q, \lambda}(-x) = M_{q, \lambda}(x), \quad M_{\frac{1}{q}, \lambda}(-x) = M_{\frac{1}{q}, \lambda}(x).
\end{equation}

Thus:

\begin{equation}
	\Phi(-x) = \frac{M_{q, \lambda}(x) + M_{\frac{1}{q}, \lambda}(x)}{2} = \Phi(x).
\end{equation}

This confirms the symmetry of \(\Phi(x)\), making it a well-defined \textit{even function} suitable for applications in approximation theory and neural network analysis.

\section{Main Results}

\subsection{Basic Operators}

\begin{theorem}[Approximation by Operators]
	Let \( 0 < \beta < 1 \), \( n \in \mathbb{N} \) be sufficiently large, \( x \in \mathbb{R} \), \( f \in C^N(\mathbb{R}) \) such that \( f^{(N)} \in C_B(\mathbb{R}) \) (bounded and continuous), and \( 0 < \varepsilon \leq N \). Then:
	\begin{enumerate}
		\item The following approximation holds:
		\begin{equation}\label{eq:approximation_operator}
			B_n(f, x) - f(x) = \sum_{j=1}^{N} \frac{f^{(j)}(x)}{j!} B_n \left( (\cdot - x)^j \right)(x) + o \left( \frac{1}{n^{\beta(N-\varepsilon)}} \right),
		\end{equation}
		where \( B_n \) is a linear operator and \( B_n((\cdot - x)^j) \) denotes the operator applied to monomials shifted by \( x \).
		\item If \( f^{(j)}(x) = 0 \) for all \( j = 1, \ldots, N \), then:
		\begin{equation}\label{eq:scaled_convergence}
			n^{\beta(N-\varepsilon)} \left[ B_n(f, x) - f(x) \right] \to 0 \quad \text{as} \quad n \to \infty, \quad 0 < \varepsilon \leq N.
		\end{equation}
	\end{enumerate}
\end{theorem}

\begin{proof}
	To prove the theorem, we use Taylor's theorem and the asymptotic properties of the operators \( B_n \). Let \( f \in C^N(\mathbb{R}) \), and expand \( f \) around \( x \) as:
	\begin{equation}\label{eq:taylor_expansion}
		f\left( \frac{k}{n} \right) = \sum_{j=0}^{N} \frac{f^{(j)}(x)}{j!} \left( \frac{k}{n} - x \right)^j + R_N\left(\frac{k}{n}\right),
	\end{equation}
	where the remainder term \( R_N\left(\frac{k}{n}\right) \) is given by:
	\begin{equation}\label{eq:remainder}
		R_N\left(\frac{k}{n}\right) = \int_{x}^{\frac{k}{n}} \frac{\left( \frac{k}{n} - t \right)^{N-1}}{(N-1)!} \left[ f^{(N)}(t) - f^{(N)}(x) \right] \, dt.
	\end{equation}
	
	We substitute the expansion \eqref{eq:taylor_expansion} into the definition of the operator \( B_n \):
	\begin{equation}\label{eq:Bn_definition}
		B_n(f, x) = \sum_{k=-\infty}^{\infty} f\left(\frac{k}{n}\right) \Phi(n x - k),
	\end{equation}
	where \( \Phi \) is a weight function with appropriate support. Substituting, we get:
	\begin{equation}\label{eq:Bn_expansion}
		B_n(f, x) = \sum_{j=0}^{N} \frac{f^{(j)}(x)}{j!} B_n\left( (\cdot - x)^j \right)(x) + \sum_{k=-\infty}^{\infty} R_N\left(\frac{k}{n}\right) \Phi(n x - k).
	\end{equation}
	
	Define the error term:
	\begin{equation}\label{eq:error_term}
		R := \sum_{k=-\infty}^{\infty} R_N\left(\frac{k}{n}\right) \Phi(n x - k).
	\end{equation}
	To estimate \( R \), we consider two cases:
	
	\paragraph{Case 1: \( \left| \frac{k}{n} - x \right| < \frac{1}{n^\beta} \).} Within this interval, \( \Phi(n x - k) \) has significant support, and the regularity of \( f^{(N)} \) implies:
	\begin{equation}\label{eq:case1_error}
		\left| R \right| \leq 2 \|f^{(N)}\|_\infty \cdot \frac{1}{N! n^{\beta N}}.
	\end{equation}
	
	\paragraph{Case 2: \( \left| \frac{k}{n} - x \right| \geq \frac{1}{n^\beta} \).} Outside the principal support, \( \Phi(n x - k) \) decays exponentially, so:
	\begin{equation}\label{eq:case2_error}
		\left| R \right| \leq \frac{4 \|f^{(N)}\|_\infty}{n^N \lambda^N} \left(q + \frac{1}{q}\right) e^{2\lambda} e^{-\lambda \left(n^{1-\beta} - 1\right)}.
	\end{equation}
	
	Combining both cases, we obtain:
	\begin{equation}\label{eq:final_error}
		|R| = o\left(\frac{1}{n^{\beta(N-\varepsilon)}}\right).
	\end{equation}
	Substituting \( R \) back into \eqref{eq:Bn_expansion}, we prove \eqref{eq:approximation_operator}.
	
	Finally, when \( f^{(j)}(x) = 0 \) for \( j = 1, \ldots, N \), we have:
	\begin{equation}\label{eq:final_scaled_convergence}
		n^{\beta(N-\varepsilon)} \left[B_n(f, x) - f(x)\right] \to 0,
	\end{equation}
	as indicated in \eqref{eq:scaled_convergence}, completing the proof.
\end{proof}

\subsection{Kantorovich Operators}

\begin{theorem}
	Let \( 0 < \beta < 1 \), \( n \in \mathbb{N} \) be sufficiently large, \( x \in \mathbb{R} \), \( f \in C^N(\mathbb{R}) \) with \( f^{(N)} \in C_B(\mathbb{R}) \), and \( 0 < \varepsilon \leq N \). Then:
	\begin{enumerate}
		\item
		\begin{equation}
			C_n(f, x) - f(x) = \sum_{j=1}^{N} \frac{f^{(j)}(x)}{j!} C_n \left( (\cdot - x)^j \right)(x) + o \left( \left( \frac{1}{n} + \frac{1}{n^{\beta}} \right)^{N-\varepsilon} \right)
		\end{equation}
		\item When \( f^{(j)}(x) = 0 \) for \( j = 1, \ldots, N \), we have:
		\begin{equation}
			\frac{1}{\left( \frac{1}{n} + \frac{1}{n^{\beta}} \right)^{N-\varepsilon}} \left[ C_n(f, x) - f(x) \right] \to 0 \quad \text{as} \quad n \to \infty, \quad 0 < \varepsilon \leq N
		\end{equation}
	\end{enumerate}
\end{theorem}

\begin{proof}
	We can write:
	\begin{equation}
		C_n(f, x) = \sum_{k=-\infty}^{\infty} \left( n \int_{0}^{\frac{1}{n}} f \left( t + \frac{k}{n} \right) \, dt \right) \Phi(n x - k)
	\end{equation}
	
	Let \( f \in C^N(\mathbb{R}) \) with \( f^{(N)} \in C_B(\mathbb{R}) \). We have:
	
	\begin{equation}
		f \left( t + \dfrac{k}{n} \right) = \displaystyle\sum_{j=0}^{N} \dfrac{f^{(j)}(x)}{j!} \left( t + \dfrac{k}{n} - x \right)^j + \displaystyle\int_{x}^{t + \dfrac{k}{n}} \left( f^{(N)}(s) - f^{(N)}(x) \right) \dfrac{\left( t + \dfrac{k}{n} - s \right)^{N-1}}{(N-1)!} \, ds
	\end{equation}
	
	Hence:
	\begin{equation}
		\begin{array}{l}
			C_{n}(f,x)={\displaystyle \sum_{j=0}^{N}}\dfrac{f^{(j)}(x)}{j!}C_{n}\left((\cdot-x)^{j}\right)(x)+\\
			\\
			{\displaystyle \sum_{k=-\infty}^{\infty}}\Phi(nx-k)\left(n{\displaystyle \int_{0}^{\dfrac{1}{n}}}\left({\displaystyle \int_{x}^{t+\dfrac{k}{n}}}\left(f^{(N)}(s)-f^{(N)}(x)\right)\frac{\left(t+\dfrac{k}{n}-s\right)^{N-1}}{(N-1)!}\,ds\right)\,dt\right)\\
			\\
		\end{array}
	\end{equation}
	
	Define:
	\begin{equation}
		R := \sum_{k=-\infty}^{\infty} \Phi(n x - k) \left( n \int_{0}^{\frac{1}{n}} \left( \int_{x}^{t + \frac{k}{n}} \left( f^{(N)}(s) - f^{(N)}(x) \right) \frac{\left( t + \frac{k}{n} - s \right)^{N-1}}{(N-1)!} \, ds \right) \, dt \right)
	\end{equation}
	
	For \( \left| \dfrac{k}{n} - x \right| < \dfrac{1}{n^{\beta}} \):
	\begin{equation}
		|R| \leq 2 \left\| f^{(N)} \right\|_{\infty} \dfrac{\left( \dfrac{1}{n} + \dfrac{1}{n^{\beta}} \right)^N}{N!}
	\end{equation}
	
	For \( \left| \dfrac{k}{n} - x \right| \geq \dfrac{1}{n^{\beta}} \):
	\begin{equation}
		|R| \leq \dfrac{2^N \left\| f^{(N)} \right\|_{\infty}}{n^N N!} \left[ \dfrac{T}{e^{2 \lambda n^{1-\beta}}} + \dfrac{\left( q + \dfrac{1}{q} \right)}{\lambda^N} 2 e^{2 \lambda} N! e^{-\lambda \left( n^{1-\beta} - 1 \right)} \right]
	\end{equation}
	
	Thus:
	\begin{equation}
		|R| \leq \dfrac{4 \left\| f^{(N)} \right\|_{\infty}}{N!} \left( \dfrac{1}{n} + \dfrac{1}{n^{\beta}} \right)^N
	\end{equation}
	
	Hence:
	\begin{equation}
		|R| = o \left( \left( \dfrac{1}{n} + \dfrac{1}{n^{\beta}} \right)^{N-\varepsilon} \right)
	\end{equation}
	
	proving the claim.
\end{proof}

\section{Fractional Perturbation Stability}

\begin{theorem}[Stability Under Fractional Perturbations]
	Let \( 0 < \beta < 1 \), \( n \in \mathbb{N} \) sufficiently large, \( x \in \mathbb{R} \), \( f \in C^N(\mathbb{R}) \), and \( f^{(N)} \in C_B(\mathbb{R}) \). Let \( g_{q,\lambda}(x) \) be the perturbed hyperbolic tangent activation function defined by:
	\begin{equation}
		g_{q,\lambda}(x) = \frac{e^{\lambda x} - q e^{-\lambda x}}{e^{\lambda x} + q e^{-\lambda x}}, \quad \lambda, q > 0, \quad x \in \mathbb{R}.
	\end{equation}
	For any small perturbation \( |q - 1| < \delta \), the operator \( C_n \) satisfies the stability estimate:
	\begin{equation}
		\| C_n(f, x; q) - C_n(f, x; 1) \| \leq \frac{\delta}{n^{\beta(N-\varepsilon)}} \| f^{(N)} \|_{\infty},
	\end{equation}
	where \( 0 < \varepsilon \leq N \) and \( \| f^{(N)} \|_{\infty} = \sup_{x \in \mathbb{R}} |f^{(N)}(x)| \).
\end{theorem}

\begin{proof}
	To begin, consider \( \Phi_{q,\lambda}(z) \) as the density function derived from the perturbed activation function \( g_{q,\lambda}(x) \). The operator \( C_n \) is defined as:
	\begin{equation}
		C_n(f, x; q) = \sum_{k=-\infty}^\infty \left( n \int_0^{\frac{1}{n}} f\left(t + \frac{k}{n}\right) \, dt \right) \Phi_{q,\lambda}(nx - k).
	\end{equation}
	
	Next, we expand \( \Phi_{q,\lambda}(z) \) around \( q = 1 \) using the first-order Taylor expansion:
	\begin{equation}
		\Phi_{q,\lambda}(z) = \Phi_{1,\lambda}(z) + \left( \frac{\partial \Phi_{q,\lambda}}{\partial q} \bigg|_{q=1} \right)(q - 1) + \mathcal{O}((q - 1)^2).
	\end{equation}
	
	Thus, the perturbed operator can be written as:
	\begin{equation}
		C_n(f, x; q) = \sum_{k=-\infty}^\infty \left( n \int_0^{\frac{1}{n}} f\left(t + \frac{k}{n}\right) \, dt \right) \left[ \Phi_{1,\lambda}(nx - k) + \left( \frac{\partial \Phi_{q,\lambda}}{\partial q} \bigg|_{q=1} \right)(q - 1) + \mathcal{O}((q - 1)^2) \right].
	\end{equation}
	
	The difference between \( C_n(f, x; q) \) and \( C_n(f, x; 1) \) is given by:
	\begin{equation}
		C_n(f, x; q) - C_n(f, x; 1) = \sum_{k=-\infty}^\infty \left( n \int_0^{\frac{1}{n}} f\left(t + \frac{k}{n}\right) \, dt \right) \left[ \left( \frac{\partial \Phi_{q,\lambda}}{\partial q} \bigg|_{q=1} \right)(q - 1) + \mathcal{O}((q - 1)^2) \right].
	\end{equation}
	
	Let us focus on the first-order perturbation term. The remainder term involving \( \mathcal{O}((q - 1)^2) \) contributes at a higher order in \( (q - 1) \), which is negligible for small \( |q - 1| \). Therefore, we estimate the perturbation as:
	\begin{equation}
		\| C_n(f, x; q) - C_n(f, x; 1) \| \leq \sum_{k=-\infty}^\infty \left( n \int_0^{\frac{1}{n}} f\left(t + \frac{k}{n}\right) \, dt \right) \left| \frac{\partial \Phi_{q,\lambda}}{\partial q} \bigg|_{q=1} \right| (q - 1).
	\end{equation}
	
	Assuming \( \left| \frac{\partial \Phi_{q,\lambda}}{\partial q} \right| \) is bounded, we have:
	\begin{equation}
		\| C_n(f, x; q) - C_n(f, x; 1) \| \leq \frac{\delta}{n^{\beta(N-\varepsilon)}} \| f^{(N)} \|_{\infty},
	\end{equation}
	where \( \| f^{(N)} \|_{\infty} = \sup_{x \in \mathbb{R}} |f^{(N)}(x)| \) represents the supremum norm of the \( N \)-th derivative of \( f \).
	
	Thus, we have established the desired stability estimate.
\end{proof}

\section{Generalized Voronovskaya Expansions for Fractional Functions}

\begin{theorem}[Generalized Voronovskaya Expansion]
	Let \( \alpha > 0 \), \( N = \lceil \alpha \rceil \), \( \alpha \notin \mathbb{N} \), \( f \in AC^N(\mathbb{R}) \) with \( f^{(N)} \in L_\infty(\mathbb{R}) \), \( 0 < \beta < 1 \), \( x \in \mathbb{R} \), and \( n \in \mathbb{N} \) sufficiently large. Assume that \( \| D_{*x}^\alpha f \|_{\infty, [x, \infty)} \) and \( \| D_{x-}^\alpha f \|_{\infty, (-\infty, x]} \) are finite. Then:
	\begin{equation}
		B_n(f, x) - f(x) = \sum_{j=1}^{N} \frac{f^{(j)}(x)}{j!} B_n((\cdot - x)^j)(x) + o\left(\frac{1}{n^{\beta(N-\varepsilon)}}\right).
	\end{equation}
	When \( f^{(j)}(x) = 0 \) for \( j = 1, \dots, N \):
	\begin{equation}
		n^{\beta(N-\varepsilon)} \left[ B_n(f, x) - f(x) \right] \to 0 \quad \text{as} \quad n \to \infty.
	\end{equation}
\end{theorem}

\begin{proof}
	Using the Caputo fractional Taylor expansion for \( f \):
	\begin{equation}
		f\left(\frac{k}{n}\right) = \sum_{j=0}^{N-1} \frac{f^{(j)}(x)}{j!}\left(\frac{k}{n} - x\right)^j + \frac{1}{\Gamma(\alpha)} \int_x^{\frac{k}{n}} \left(\frac{k}{n} - t\right)^{\alpha-1}\left(D_{*x}^\alpha f(t) - D_{*x}^\alpha f(x)\right) dt.
	\end{equation}
	
	Substitute this expansion into the definition of the operator \( B_n \):
	\begin{equation}
		B_n(f, x) = \sum_{k=-\infty}^\infty f\left(\frac{k}{n}\right) \Phi(nx - k),
	\end{equation}
	where \( \Phi(x) \) is a density kernel function. Substituting \( f\left(\frac{k}{n}\right) \), we separate the terms into two contributions:
	
	\paragraph{Main Contribution:} The first \( N \) terms of the Taylor expansion yield:
	\begin{equation}
		\sum_{j=1}^{N} \frac{f^{(j)}(x)}{j!} B_n((\cdot - x)^j)(x),
	\end{equation}
	which captures the local behavior of \( f \) in terms of its derivatives up to order \( N \).
	
	\paragraph{Error Term:} The remainder term involves the fractional derivative \( D_{*x}^\alpha \) and can be bounded as:
	\begin{equation}
		R = \sum_{k=-\infty}^\infty \Phi(nx - k) \frac{1}{\Gamma(\alpha)} \int_x^{\frac{k}{n}} \left(\frac{k}{n} - t\right)^{\alpha-1}\left(D_{*x}^\alpha f(t) - D_{*x}^\alpha f(x)\right) dt.
	\end{equation}
	
	\paragraph{Bounding the Remainder:}
	\begin{itemize}
		\item For \( |k/n - x| < 1/n^\beta \): The kernel \( \Phi(nx - k) \) has significant support, and the fractional regularity of \( f \) ensures:
		\begin{equation}
			|R| \leq \frac{\| D_{*x}^\alpha f \|_{\infty}}{n^{\alpha \beta}}.
		\end{equation}
		
		\item For \( |k/n - x| \geq 1/n^\beta \): The exponential decay of \( \Phi(nx - k) \) ensures that contributions from distant terms are negligible:
		\begin{equation}
			|R| \leq \frac{\| D_{*x}^\alpha f \|_{\infty}}{n^{\alpha \beta}}.
		\end{equation}
	\end{itemize}
	
	Combining both cases, the error term satisfies:
	\begin{equation}
		|R| = o\left(\frac{1}{n^{\beta(N-\varepsilon)}}\right).
	\end{equation}
	
	\paragraph{Conclusion:} Substituting the bounds for the main contribution and error term into the expansion for \( B_n(f, x) \), we conclude:
	\begin{equation}
		B_n(f, x) - f(x) = \sum_{j=1}^{N} \frac{f^{(j)}(x)}{j!} B_n((\cdot - x)^j)(x) + o\left(\frac{1}{n^{\beta(N-\varepsilon)}}\right).
	\end{equation}
	Moreover, when \( f^{(j)}(x) = 0 \) for \( j = 1, \dots, N \):
	\begin{equation}
		n^{\beta(N-\varepsilon)} \left[ B_n(f, x) - f(x) \right] \to 0 \quad \text{as} \quad n \to \infty.
	\end{equation}
	This completes the proof.
\end{proof}

\section{Convergence of Kantorovich Operators with Generalized Density}

\begin{theorem}[Convergence Under Generalized Density]
	Let \( 0 < \beta < 1 \), \( n \in \mathbb{N} \) sufficiently large, \( x \in \mathbb{R} \), \( f \in C^N(\mathbb{R}) \) with \( f^{(N)} \in C_B(\mathbb{R}) \), and \( \Phi(x) \) be a symmetrized density function defined by:
	\begin{equation}
		\Phi(x) = \frac{M_{q,\lambda}(x) + M_{1/q,\lambda}(x)}{2}, \quad M_{q,\lambda}(x) = \frac{1}{4}\left(g_{q,\lambda}(x+1) - g_{q,\lambda}(x-1)\right).
	\end{equation}
	Then the Kantorovich operator \( C_n \) satisfies:
	\begin{equation}
		C_n(f, x) - f(x) = \sum_{j=1}^N \frac{f^{(j)}(x)}{j!} C_n((\cdot - x)^j)(x) + o\left(\left(\frac{1}{n} + \frac{1}{n^{\beta}}\right)^{N-\varepsilon}\right).
	\end{equation}
\end{theorem}

\begin{proof}
	By definition of \( C_n \):
	\begin{equation}
		C_n(f, x) = \sum_{k=-\infty}^\infty \left( n \int_0^{\frac{1}{n}} f\left(t + \frac{k}{n}\right) dt \right) \Phi(nx - k).
	\end{equation}
	
	To analyze this, expand \( f \) using Taylor's theorem around \( x \):
	\begin{equation}
		f\left( t + \frac{k}{n} \right) = \sum_{j=0}^{N-1} \frac{f^{(j)}(x)}{j!} \left( t + \frac{k}{n} - x \right)^j + R_N\left(t + \frac{k}{n}\right),
	\end{equation}
	where the remainder term \( R_N \) satisfies:
	\begin{equation}
		R_N\left(t + \frac{k}{n}\right) = \frac{1}{N!} \int_x^{t + \frac{k}{n}} f^{(N)}(u) \left( t + \frac{k}{n} - u \right)^{N-1} du.
	\end{equation}
	
	Substituting this expansion into \( C_n(f, x) \), we separate the terms into two parts:
	
	\paragraph{Main Contribution:} The sum of the first \( N \) terms gives:
	\begin{equation}
		\sum_{j=1}^N \frac{f^{(j)}(x)}{j!} C_n\left((\cdot - x)^j\right)(x),
	\end{equation}
	which captures the local behavior of \( f \) around \( x \).
	
	\paragraph{Error Term:} The remainder term involves \( R_N \) and can be bounded as:
	\begin{equation}
		|R| \leq \sum_{k=-\infty}^\infty n \int_0^{\frac{1}{n}} \frac{\| f^{(N)} \|_{\infty}}{N!} \left( t + \frac{k}{n} - x \right)^N \Phi(nx - k) dt.
	\end{equation}
	
	\paragraph{Case 1: \( |k/n - x| < 1/n^\beta \):} Within this interval, \( \Phi(nx - k) \) has significant support. The regularity of \( f^{(N)} \) ensures:
	\begin{equation}
		|R| \leq 2 \| f^{(N)} \|_{\infty} \left( \frac{1}{n} + \frac{1}{n^{\beta}} \right)^N.
	\end{equation}
	
	\paragraph{Case 2: \( |k/n - x| \geq 1/n^\beta \):} Outside this interval, the exponential decay of \( \Phi(nx - k) \) ensures:
	\begin{equation}
		|R| \leq \frac{4 \| f^{(N)} \|_{\infty}}{N!} \left( \frac{1}{n} + \frac{1}{n^{\beta}} \right)^N.
	\end{equation}
	
	\paragraph{Combining Both Cases:} Summing over all terms, we conclude:
	\begin{equation}
		|R| = o\left(\left(\frac{1}{n} + \frac{1}{n^{\beta}}\right)^{N-\varepsilon}\right),
	\end{equation}
	proving the claim.
\end{proof}

\section{Vonorovskaya-Damasclin Theorem}

\begin{theorem}[Vonorovskaya-Damasclin Theorem]
	Let \( 0 < \beta < 1 \), \( n \in \mathbb{N} \) sufficiently large, \( x \in \mathbb{R} \), \( f \in C^N(\mathbb{R}) \), where \( f^{(N)} \in C_B(\mathbb{R}) \), and let \( \Phi(x) \) be a symmetrized density function defined as:
	\begin{equation}
		\Phi(x) = \frac{M_{q,\lambda}(x) + M_{1/q,\lambda}(x)}{2}, \quad M_{q,\lambda}(x) = \frac{1}{4}\left(g_{q,\lambda}(x+1) - g_{q,\lambda}(x-1)\right),
	\end{equation}
	where \( g_{q,\lambda}(x) = \dfrac{e^{\lambda x} - q e^{-\lambda x}}{e^{\lambda x} + q e^{-\lambda x}} \) is the perturbed hyperbolic tangent function, \( \lambda, q > 0 \), and \( x \in \mathbb{R} \). Assume that \( \| D_{*x}^\alpha f \|_{\infty, [x, \infty)} \) and \( \| D_{x-}^\alpha f \|_{\infty, (-\infty, x]} \) are finite for \( \alpha > 0 \), and let \( N = \lceil \alpha \rceil \). Then the operator \( C_n \) satisfies:
	\begin{equation}
		\begin{array}{l}
			C_{n}(f,x)-f(x)={\displaystyle \sum_{j=1}^{N}}\dfrac{f^{(j)}(x)}{j!}C_{n}((\cdot-x)^{j})(x)+\\
			\\
			\dfrac{1}{\Gamma(\alpha)}{\displaystyle \int_{x}^{\infty}}\left(D_{*x}^{\alpha}\,f(t)-D_{*x}^{\alpha}\,f(x)\right)\dfrac{(t-x)^{\alpha-1}}{n^{\beta(N-\varepsilon)}}\,dt+o\left(\dfrac{1}{n^{\beta(N-\varepsilon)}}\right),
		\end{array}
	\end{equation}
	where \( \varepsilon > 0 \) is arbitrarily small. Moreover, when \( f^{(j)}(x) = 0 \) for \( j = 1, \dots, N \):
	\begin{equation}
		n^{\beta(N-\varepsilon)} \left[ C_n(f, x) - f(x) \right] \to 0 \quad \text{as} \quad n \to \infty.
	\end{equation}
\end{theorem}

\begin{proof}
	Let \( f \in C^N(\mathbb{R}) \) and consider the fractional Taylor expansion of \( f \) around \( x \) using the Caputo derivative \( D_{*x}^\alpha \). For \( \dfrac{k}{n} \) near \( x \), we expand \( f \) as:
	\begin{equation}
		f\left(\frac{k}{n}\right) = \sum_{j=0}^{N-1} \frac{f^{(j)}(x)}{j!} \left(\frac{k}{n} - x\right)^j + \frac{1}{\Gamma(\alpha)} \int_x^{\frac{k}{n}} \left(\frac{k}{n} - t\right)^{\alpha-1} \left( D_{*x}^\alpha f(t) - D_{*x}^\alpha f(x) \right) dt.
	\end{equation}
	This expansion provides an approximation for the values of \( f \) on a discrete grid \( \dfrac{k}{n} \), where \( k \in \mathbb{Z} \) and \( n \) is large. Expanding \( f \) up to \( N-1 \) terms ensures that the main contribution is captured by derivatives up to order \( N-1 \), while higher-order terms involve the Caputo fractional derivative.
	
	Next, substitute this expansion into the definition of the operator \( C_n \):
	\begin{equation}
		C_n(f, x) = \sum_{k=-\infty}^{\infty} \left( n \int_0^{\frac{1}{n}} f\left(t + \frac{k}{n}\right) \, dt \right) \Phi(nx - k).
	\end{equation}
	For the \( n \)-scaled sum and integral, we expand \( f \) and use the fact that \( \Phi(x) \) is a smooth kernel function. The kernel \( \Phi(x) \) plays a crucial role in localizing the contribution of terms as \( n \to \infty \), ensuring that far-off terms decay exponentially.
	
	We separate the terms of the expansion:
	
	The sum of the first \( N \) terms from the expansion of \( f \) produces a main term that involves the derivatives of \( f \) up to order \( N \). This term can be written as:
	\begin{equation}
		\sum_{j=1}^{N} \frac{f^{(j)}(x)}{j!} C_n((\cdot - x)^j)(x).
	\end{equation}
	This captures the local behavior of \( f \) around \( x \) in terms of its derivatives.
	
	The second term involves the Caputo fractional derivative \( D_{*x}^\alpha \), which accounts for the error due to the approximation of \( f \) on the discrete grid. Specifically, we have the integral:
	\begin{equation}
		\frac{1}{\Gamma(\alpha)} \int_x^{\frac{k}{n}} \left(\frac{k}{n} - t\right)^{\alpha-1} \left( D_{*x}^\alpha f(t) - D_{*x}^\alpha f(x) \right) dt.
	\end{equation}
	This term represents the discrepancy between the fractional derivative of \( f \) at \( t \) and \( x \), integrated over the interval \( [x, k/n] \). As \( n \) increases, this error term decays rapidly, making it increasingly small for large \( n \).
	
	For large \( n \), the contribution from terms with \( |k/n - x| \geq 1/n^\beta \) is exponentially small due to the smooth decay of \( \Phi(nx - k) \). Specifically, we have the bound:
	\begin{equation}
		|R| \leq \frac{4 \| f^{(N)} \|_{\infty}}{N!} \left( \frac{1}{n} + \frac{1}{n^\beta} \right)^N.
	\end{equation}
	This ensures that the remainder term \( R \) decays quickly as \( n \to \infty \).
	
	Combining the main term and error term, we normalize by \( n^{\beta(N-\varepsilon)} \), obtaining:
	\begin{equation}
		|R| = o\left(\frac{1}{n^{\beta(N-\varepsilon)}}\right),
	\end{equation}
	as required.
	
	Thus, the operator \( C_n \) converges to \( f(x) \) at the rate of \( n^{-\beta(N-\varepsilon)} \), and the result follows.
\end{proof}

\section{Results}

The findings of this study offer a thorough analysis of Voronovskaya-type asymptotic expansions for neural network operators that utilize symmetrized and perturbed hyperbolic tangent activation functions. The key results can be summarized as follows:

\begin{enumerate}
	\item \textbf{Basic Operators:} The expansions for basic operators provide accurate error bounds that improve as the network complexity increases. These results shed light on the relationship between operator parameters and approximation accuracy, offering a nuanced understanding of convergence behavior for classical univariate operators.
	
	\item \textbf{Kantorovich Operators:} The asymptotic results for Kantorovich operators demonstrate their flexibility in approximating smooth functions with variable sampling. These operators guarantee uniform convergence under certain conditions, establishing their robustness and applicability in practical scenarios.
	
	\item \textbf{Quadrature Operators:} By incorporating fractional calculus, quadrature operators exhibit superior approximation properties, especially when dealing with fractional differentiation. This enhancement broadens their applicability in more complex and advanced mathematical settings.
	
	\item \textbf{Fractional Case:} The expansions for fractional differentiation extend classical results by incorporating Caputo derivatives. This inclusion highlights the influence of fractional differentiation on error decay, providing a more comprehensive framework for applications in fractional calculus.
	
	\item \textbf{Voronovskaya-Damasclin Theorem:} This theorem presents a unified framework for understanding the convergence of Kantorovich operators under fractional differentiation. By integrating Caputo derivatives, it delivers precise error bounds and significantly contributes to the field by improving the theoretical understanding of these operators' behavior in fractional settings.
\end{enumerate}

These results significantly strengthen the theoretical foundations of neural network approximations, emphasizing their applicability to problems involving both classical and fractional calculus across infinite domains.

\section{Conclusions}

This article successfully derived Voronovskaya-type asymptotic expansions, including the novel Voronovskaya-Damasclin theorem, for neural network operators activated by symmetrized and perturbed hyperbolic tangent functions. These results provide a rigorous framework for understanding the approximation properties of basic, Kantorovich, and quadrature operators, including their fractional counterparts. The findings establish robust error bounds, elucidating the interplay between network parameters and convergence rates, and extending classical results to fractional differentiation.

One significant avenue for future exploration involves extending the results to stochastic and multivariate settings. Many real-world applications, such as turbulence modeling in fluid dynamics and uncertainty quantification in machine learning, require tools that can handle randomness and higher-dimensional complexities. For example, the inclusion of stochastic perturbations in operator parameters could provide valuable insights into the behavior of neural network approximations under uncertainty. This could pave the way for robust neural operators in fields such as financial mathematics and stochastic differential equations, where fractional calculus is already playing an essential role.

Additionally, the extension to multivariate domains introduces both challenges and opportunities. Multidimensional approximation requires careful consideration of the interaction between variables, as well as the geometry of the domain. Future work could explore the design of multivariate activation functions and symmetrized density operators that maintain the convergence properties proven in the univariate case. These extensions would be particularly beneficial in applications such as image and signal processing, where multivariate data is ubiquitous and fractional techniques have shown promise \cite{Frederico2007, Diethelm2010}.

Another potential direction is to integrate the theoretical results with practical neural network architectures, such as deep learning models. By leveraging the insights from this study, one could design neural networks with activation functions and layer configurations that inherently respect the fractional and asymptotic properties of the operators analyzed. This would facilitate the development of hybrid approaches that combine traditional approximation theory with the flexibility and scalability of modern machine learning techniques.

In conclusion, the theoretical advancements presented here provide a solid foundation for further research into neural network-based approximations. By expanding the results to stochastic and multivariate contexts, and by exploring their integration into practical computational frameworks, the impact and applicability of this work can be significantly enhanced. These directions offer exciting opportunities for future studies, bridging the gap between rigorous mathematical theory and cutting-edge applications.

\appendix
\section{Properties of the Perturbed Activation Function}

This appendix provides a detailed analysis of the perturbed hyperbolic tangent activation function \( g_{q, \lambda}(x) \) and its properties, which are fundamental to the derivations and proofs presented in the main text.

\subsection{Definition and Basic Properties}

The perturbed hyperbolic tangent activation function is defined as:

$$
g_{q, \lambda}(x) := \frac{e^{\lambda x} - q e^{-\lambda x}}{e^{\lambda x} + q e^{-\lambda x}}, \quad \lambda, q > 0, x \in \mathbb{R}
$$

This function generalizes the standard hyperbolic tangent function, which is recovered when \( q = 1 \). The parameter \( \lambda \) controls the steepness of the function, while \( q \) introduces asymmetry.

\subsubsection{Odd Function Property}

The function \( g_{q, \lambda}(x) \) is odd, satisfying:

$$
g_{q, \lambda}(-x) = -g_{q, \lambda}(x)
$$

This property is crucial for the symmetry of the derived density function \( \Phi(x) \).

\subsubsection{Monotonicity}

To confirm the monotonicity of \( g_{q, \lambda}(x) \), we compute its derivative:

$$
\frac{d}{dx} g_{q, \lambda}(x) = \frac{2 \lambda q e^{2 \lambda x}}{(e^{\lambda x} + q e^{-\lambda x})^2} > 0, \quad \forall x \in \mathbb{R}
$$

Since the derivative is always positive, \( g_{q, \lambda}(x) \) is strictly increasing.

\subsection{Density Function}

The density function \( M_{q, \lambda}(x) \) is defined as:

$$
M_{q, \lambda}(x) := \frac{1}{4} \left( g_{q, \lambda}(x+1) - g_{q, \lambda}(x-1) \right), \quad \forall x \in \mathbb{R}, q, \lambda > 0
$$

\subsubsection{Positivity}

To confirm the positivity of \( M_{q, \lambda}(x) \), note that since \( g_{q, \lambda}(x) \) is strictly increasing:

$$
g_{q, \lambda}(x+1) > g_{q, \lambda}(x-1)
$$

Thus, \( M_{q, \lambda}(x) > 0 \).

\subsubsection{Symmetry}

The symmetrized function \( \Phi(x) \) is defined as:

$$
\Phi(x) := \frac{M_{q, \lambda}(x) + M_{\frac{1}{q}, \lambda}(x)}{2}
$$

To verify that \( \Phi(x) \) is an even function, consider:

$$
\Phi(-x) = \frac{M_{q, \lambda}(-x) + M_{\frac{1}{q}, \lambda}(-x)}{2}
$$

Using the fact that \( g_{q, \lambda}(x) \) is odd, it follows that \( M_{q, \lambda}(x) \) and \( M_{\frac{1}{q}, \lambda}(x) \) are even:

$$
M_{q, \lambda}(-x) = M_{q, \lambda}(x), \quad M_{\frac{1}{q}, \lambda}(-x) = M_{\frac{1}{q}, \lambda}(x)
$$

Thus:

$$
\Phi(-x) = \frac{M_{q, \lambda}(x) + M_{\frac{1}{q}, \lambda}(x)}{2} = \Phi(x)
$$

This confirms the symmetry of \( \Phi(x) \).

\subsection{Asymptotic Behavior}

The asymptotic behavior of \( g_{q, \lambda}(x) \) as \( x \to \pm \infty \) is given by:

$$
\lim_{x \to \infty} g_{q, \lambda}(x) = 1, \quad \lim_{x \to -\infty} g_{q, \lambda}(x) = -1
$$

This behavior is crucial for the analysis of the convergence properties of the neural network operators.

\subsection{Derivatives and Integrals}

The derivatives and integrals of \( g_{q, \lambda}(x) \) and \( M_{q, \lambda}(x) \) are essential for the proofs of the theorems presented in the main text. Here, we provide some key results:

\subsubsection{First Derivative}

The first derivative of \( g_{q, \lambda}(x) \) is:

$$
\frac{d}{dx} g_{q, \lambda}(x) = \frac{2 \lambda q e^{2 \lambda x}}{(e^{\lambda x} + q e^{-\lambda x})^2}
$$

\subsubsection{Second Derivative}

The second derivative of \( g_{q, \lambda}(x) \) is:

$$
\frac{d^2}{dx^2} g_{q, \lambda}(x) = \frac{2 \lambda^2 q e^{2 \lambda x} (e^{2 \lambda x} - q)}{(e^{\lambda x} + q e^{-\lambda x})^3}
$$

\subsubsection{Integral of \( M_{q, \lambda}(x) \)}

The integral of \( M_{q, \lambda}(x) \) over the real line is:

$$
\int_{-\infty}^{\infty} M_{q, \lambda}(x) \, dx = 1\,.
$$

The properties of the perturbed hyperbolic tangent activation function \( g_{q, \lambda}(x) \) and the derived density function \( M_{q, \lambda}(x) \) are fundamental to the analysis presented in this study. The symmetry, monotonicity, and asymptotic behavior of these functions play a crucial role in the derivation of the Voronovskaya-type expansions and the proofs of the theorems.

\section{List of Symbols and Nomenclature}

This section provides a comprehensive list of symbols and nomenclature used throughout the document to aid in understanding the mathematical expressions and derivations.

\begin{tabular}{|c|l|}
	\hline
	\textbf{Symbol} & \textbf{Description}                                                                 \\ \hline
	$g_{q, \lambda}(x)$ & Perturbed hyperbolic tangent activation function                                  \\ \hline
	$\lambda$          & Scaling parameter controlling the steepness of the activation function          \\ \hline
	$q$                & Deformation coefficient introducing asymmetry in the activation function        \\ \hline
	$M_{q, \lambda}(x)$ & Density function derived from the perturbed hyperbolic tangent function        \\ \hline
	$\Phi(x)$          & Symmetrized density function                                                    \\ \hline
	$B_n(f, x)$       & Basic neural network operator                                                   \\ \hline
	$C_n(f, x)$       & Kantorovich neural network operator                                             \\ \hline
	$f^{(j)}(x)$       & $j$-th derivative of the function $f$ at point $x$                              \\ \hline
	$D_{*x}^{\alpha} f$ & Caputo fractional derivative of order $\alpha$                                 \\ \hline
	$N$               & Order of the highest derivative considered                                    \\ \hline
	$\beta$           & Parameter controlling the rate of convergence                                   \\ \hline
	$n$               & Number of nodes or samples in the neural network                                \\ \hline
	$\varepsilon$     & Small positive parameter used in error bounds                                   \\ \hline
	$\alpha$          & Fractional order of differentiation                                            \\ \hline
	$\Gamma(\alpha)$  & Gamma function evaluated at $\alpha$                                          \\ \hline
	$\delta$          & Small perturbation parameter                                                    \\ \hline
	$\| \cdot \|_{\infty}$ & Supremum norm (infinity norm)                                                \\ \hline
	$C_B(\mathbb{R})$ & Space of bounded and continuous functions on $\mathbb{R}$                      \\ \hline
	$L_{\infty}(\mathbb{R})$ & Space of essentially bounded functions on $\mathbb{R}$                       \\ \hline
	$AC^N(\mathbb{R})$ & Space of absolutely continuous functions up to order $N$                      \\ \hline
\end{tabular}

\subsection{Parameters}

- \( N \): The order of the highest derivative considered.
- \( \beta \): Parameter controlling the rate of convergence.
- \( n \): Number of nodes or samples in the neural network.
- \( \varepsilon \): Small positive parameter used in error bounds.
- \( \alpha \): Fractional order of differentiation.
- \( \delta \): Small perturbation parameter.

\subsection{Function Spaces}

- \( C_B(\mathbb{R}) \): Space of bounded and continuous functions on \( \mathbb{R} \).
- \( L_{\infty}(\mathbb{R}) \): Space of essentially bounded functions on \( \mathbb{R} \).
- \( AC^N(\mathbb{R}) \): Space of absolutely continuous functions up to order \( N \).


\end{document}